\documentclass[fleqn]{article}
\usepackage{graphicx,color} % Required for inserting images
\usepackage{amsthm,amssymb,amsmath}
\usepackage[all]{xy}
\newtheorem{theorem}{Theorem}
\newtheorem{lemma}{Lemma}
\newtheorem{proposition}{Proposition}
\theoremstyle{remark}
\newtheorem{remark}{Remark}

\newcommand{\Z}{\mathbb{Z}}
\newcommand{\C}{\mathbb{C}}
\newcommand{\R}{\mathbb{R}}
\newcommand{\Q}{\mathbb{Q}}
\newcommand{\B}{\mathcal{B}}
\newcommand{\V}{\mathcal{V}}
\newcommand{\PP}{\mathcal{P}}
\newcommand{\QQ}{\mathcal{Q}}
\newcommand{\CC}{\mathcal{C}}
\renewcommand{\O}{\mathcal{O}}
\renewcommand{\H}{\mathbb{H}}
\renewcommand{\P}{\mathbb{P}}
\newcommand{\PB}{\mathcal{P}}
\newcommand{\SL}{\mathrm{SL}}

\renewcommand{\epsilon}{\varepsilon}
\newcommand{\deuxp}{\!:\!}
\DeclareMathOperator{\Vol}{Vol}

\title{A formula for the volume of two-bridge knots}
\author{Julien Marché}
\begin{document}
\maketitle
\begin{abstract}
We give a closed formula for the volume of a two-bridge knot, more precisely for its Bloch invariant. We obtain this formula without triangulating the complement: instead, we derive it from the Hopf formula for the second homology of the fundamental group of the complement and a systematic use of Fox derivatives.
\end{abstract}

\section{Introduction}

Let $p,q$ be two coprime odd integers satisfying $0<q<p$. We denote by $K(p,q)\subset S^3$ the two-bridge knot with parameter $(p,q)$ and recall its main properties, see \cite{BZH}. 
\begin{enumerate}
    \item The double cover of $S^3$ ramified along $K(p,q)$ is the lens space $L(p,q)$.
    \item The knots $K(p,q)$ and $K(p',q')$ are isotopic (up to orientation) if and only if $p'=p$ and $q'=\pm q^{\pm 1}\mod p$. 
    \item The knot $K(p,q)$ is hyperbolic if and only if $q\ne 1$. 
\end{enumerate}

We define a sequence of signs $\epsilon_n$ by the formula $\epsilon_n=(-1)^{\lfloor nq/p\rfloor}$ and a sequence of polynomials $P_n,Q_n\in\Z[x]$ of respective degrees $n$ and $n-1$ by the formula 
$$\frac{P_n(x)}{Q_n(x)}=\epsilon_1 x+
\cfrac{1}
{\epsilon_2 x+\cfrac{1}{\cdots+\cfrac{1}{\epsilon_{n}x}}
}.$$

We denote by $D:\C\to \R$ the Bloch-Wigner dilogarithm: $D(z)$ is the volume of the ideal hyperbolic tetrahedron whose vertices in $\P^1(\C)=\partial \H^3$ are $\infty,0,1,z$, see \cite{dilog}. Finally, we set $\ell=\ell(p,q)$ to be the unique (odd) integer satisfying $0<\ell<2p$ and congruent to $-q^{-1}$ modulo $2p$. 
%Given $x\in \C$, we set $v(z)=\frac{1}{x+\frac{1}{z}}$ and $z_j=\frac{P_j(x)}{Q_j(x)}$. 
\begin{theorem}\label{principal}
Let $0<q<p$ be as above and set $Z_{p,q}=\{x\in \C, P_{p-1}(x)=0\}$. For any $x\in Z_{p,q}$ and $n>0$, we define $z_n=P_n(x)/Q_n(x)$ and set $z_0=\infty$. 
The volume $V(p,q)$ of $S^3\setminus K(p,q)$ is given by the following formula:
\begin{equation*}
\begin{split}
V(p,q)=\max_{x\in Z_{p,q}}\Big\{\sum_{j=1}^{\frac{p-1}{2}}\Big(
D\big(\frac{z_{\ell-1}z_{2p-2}-z_{2j-2}z_{2p-2}+z_{2j-2}z_{\ell-1}}{z_{\ell-1}z_{2p-2}-z_{2j}z_{2p-2}+z_{2j}z_{\ell-1}}\big)+D\big(\frac{z_{2j-2}}{z_{2j}}\big)
\\
+D\big(\frac{z_{2j-2}(z_{2p-2}-z_{2j})}{z_{2j}(z_{2p-2}-z_{2j-2})}\big)
+D\big(\frac{(z_{2p-2}-z_{2j-2})(z_{\ell-1}-z_{2j})}{(z_{2p-2}-z_{2j})(z_{\ell-1}-z_{2j-2})}\big)
\Big)\Big\}.
\end{split}
\end{equation*}
%\begin{equation*}
%\begin{split}
%\sum_{j=1}^{\frac{p-1}{2}}\Big(
%D\big(\frac{v(z_{k-1})-z_{2j-2}}{v(z_{k-1})-z_{2j}}\big)+D\big(\frac{z_{2j-2}}{z_{2j}}\big)
%+D\big(\frac{z_{2j-2}(z_{2p-2}-z_{2j})}{z_{2j}(z_{2p-2}-z_{2j-2})}\big)
%\\
%+D\big(\frac{(z_{2p-2}-z_{2j-2})(z_{k-1}-z_{2j})}{(z_{2p-2}-z_{2j})(z_{k-1}-z_{2j-2})}\big)
%\Big)
%\end{split}
%\end{equation*}
\end{theorem}

There is a Bloch group version of the theorem: indeed, suppose $q>1$ and let $k$ be the trace field of $K(p,q)$, which is generated by $x^2$ for some root $x$ of $P_p$. There is a representation $\rho:\pi_1(S^3\setminus K(p,q))\to \SL_2(k)$ whose restriction to the boundary is parabolic, defined by Riley in \cite{Riley}. It defines a Bloch invariant $\beta_{p,q}\in \B(k)$ which recovers the volume: this invariant is given by the same formula as in Theorem \ref{principal} without the max and with $D(z)$ replaced by $[z]$, see Section \ref{Bloch}.

As an illustration, we plot in Figure \ref{feu} all pairs $(\frac{q}{p}, V(p,q))$ for $p<50$.

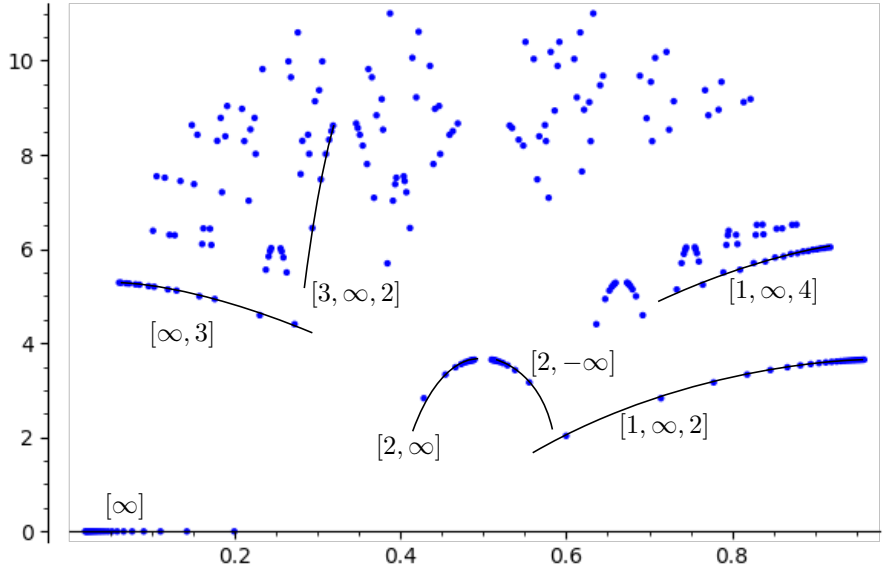
\begin{figure}[htbp]
    \centering
    \def\svgwidth{12cm}
    %% Creator: Inkscape 1.1 (c4e8f9e, 2021-05-24), www.inkscape.org
%% Accompanies image file 'feu.pdf' (pdf, eps, ps)
\begingroup%
  \makeatletter%
  \providecommand\color[2][]{%
    \errmessage{(Inkscape) Color is used for the text in Inkscape, but the package 'color.sty' is not loaded}%
    \renewcommand\color[2][]{}%
  }%
  \providecommand\transparent[1]{%
    \errmessage{(Inkscape) Transparency is used (non-zero) for the text in Inkscape, but the package 'transparent.sty' is not loaded}%
    \renewcommand\transparent[1]{}%
  }%
  \providecommand\rotatebox[2]{#2}%
  \newcommand*\fsize{\dimexpr\f@size pt\relax}%
  \newcommand*\lineheight[1]{\fontsize{\fsize}{#1\fsize}\selectfont}%
  \ifx\svgwidth\undefined%
    \setlength{\unitlength}{424.80000173bp}%
    \ifx\svgscale\undefined%
      \relax%
    \else%
      \setlength{\unitlength}{\unitlength * \real{\svgscale}}%
    \fi%
  \else%
    \setlength{\unitlength}{\svgwidth}%
  \fi%
  \global\let\svgwidth\undefined%
  \global\let\svgscale\undefined%
  \makeatother%
  \begin{picture}(1,0.65932203)%
    \lineheight{1}%
    \setlength\tabcolsep{0pt}%
    \put(0,0){\includegraphics[width=\unitlength,page=1]{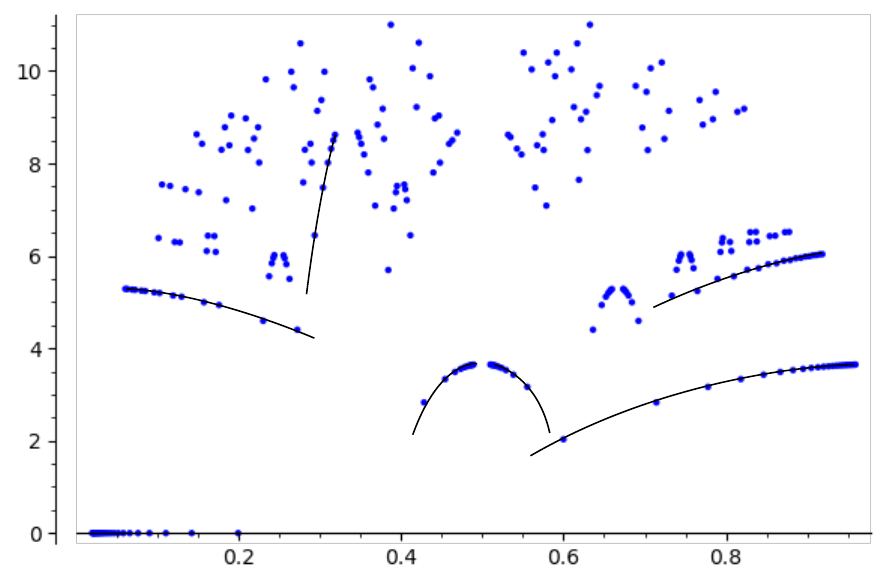}}%
    \put(0.17465391,0.26815288){\color[rgb]{0,0,0}\makebox(0,0)[lt]{\lineheight{1.25}\smash{\begin{tabular}[t]{l}$[\infty,3]$\end{tabular}}}}%
    \put(0.35304584,0.31211461){\makebox(0,0)[lt]{\lineheight{1.25}\smash{\begin{tabular}[t]{l}$[3,\infty,2]$\end{tabular}}}}%
    \put(0.69464633,0.16623887){\makebox(0,0)[lt]{\lineheight{1.25}\smash{\begin{tabular}[t]{l}$[1,\infty,2]$\end{tabular}}}}%
    \put(0.42571893,0.14329911){\makebox(0,0)[lt]{\lineheight{1.25}\smash{\begin{tabular}[t]{l}$[2,\infty]$\end{tabular}}}}%
    \put(0.59608077,0.23579475){\makebox(0,0)[lt]{\lineheight{1.25}\smash{\begin{tabular}[t]{l}$[2,-\infty]$\end{tabular}}}}%
    \put(0.8157132,0.3159683){\makebox(0,0)[lt]{\lineheight{1.25}\smash{\begin{tabular}[t]{l}$[1,\infty,4]$\end{tabular}}}}%
    \put(0.12503634,0.07780387){\makebox(0,0)[lt]{\lineheight{1.25}\smash{\begin{tabular}[t]{l}$[\infty]$\end{tabular}}}}%
  \end{picture}%
\endgroup%

    \caption{Volumes of two-bridge knots.}
    \label{feu}
\end{figure}

As a subset of $\R^2$, it has many accumulation points that we materialize using black lines. Recall that writing $\frac{p}{q}$ as a continued fraction 
$$\frac{p}{q}=a_1+\cfrac{1}{a_2+\cfrac{1}{\cdots+\frac{1}{a_n}}}=[a_1,\ldots,a_n]\textrm{ for }a_1,\ldots,a_n\in \Z\setminus\{0\}$$
yields a diagram for $K(p,q)$ in the Conway form given (for odd $n$) in Figure \ref{conway}.
\begin{figure}[htbp]
\begin{center}
\includegraphics[width=12cm]{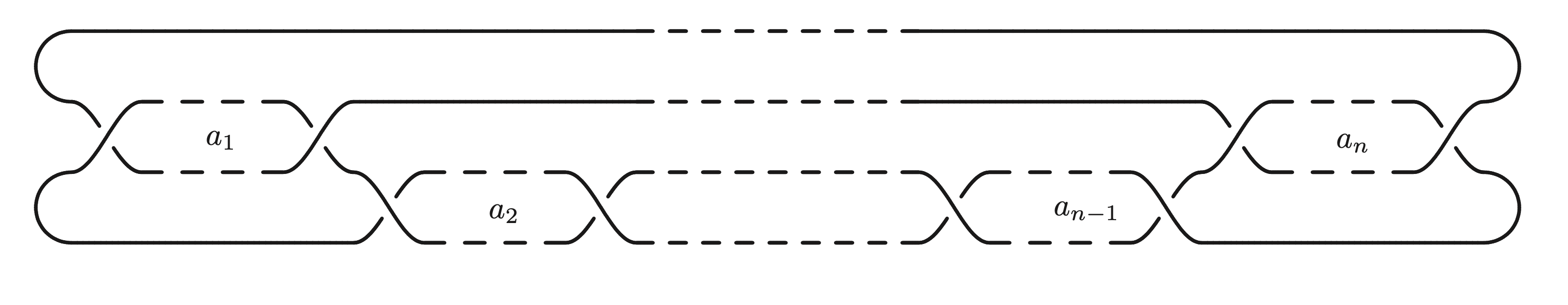}
\end{center}
\caption{Conway normal form}
\label{conway}
\end{figure}

The number of twists is denoted by the integer $|a_i|$, and the convention of sign of $a_i$ differs depending on the parity of $i$. In the picture, all $a_i$ are positive. Indeed, there is a unique continued fraction with positive entries yielding $q/p$ which provides an alternating diagram. Using Lackenby's bound \cite{Lackenby}

$$v_3\frac{n-2}{2}\le \Vol(p,q)\le 16v_3(n-1)$$
where $v_3=D(e^{i\pi/3})$, we know that $V(p,q)$ is roughly proportional to the depth of the positive continued fraction expressing $p/q$. In particular, it is unbounded.

Take any continued fraction of the form $[a_1,\ldots,a_n]$, and suppose that the integers $a_{i_1},\ldots,a_{i_k}$ go to infinity for some $1\le i_1<\cdots<i_k\le n$. The corresponding knot converges to the complement of a hyperbolic link with $k+1$ components, obtained by circling each of the twist regions of index $i_1,\ldots,i_k$ with a trivial knot. The continued fraction also converges, which explains the accumulation points observed in Figure \ref{feu}: we parametrize them by writing $\infty$ in place of $a_{i_1},\ldots,a_{i_k}$. For instance, the line $[\infty]$ correspond to the fractions $\frac{n}{1}$ for odd $n$, hence to the knots $K(n,1)$ which are torus knots, hence with trivial volume. The line $[1,\infty,2]$ correspond to $1+\frac{1}{n+1/2}$ hence to the knots $K(2n+3,2n+1)$ which are twist knots.

There are also accumulation points of accumulation points (and further) but Lackenby's bound show that they appear higher and higher in the picture. It would be interesting to study the coarsest topology on $\Q$ making the map $q/p\mapsto K(p,q)$ continuous.

Let us now comment on the technique of proof. It relies mainly on the simple presentation of $G=\pi_1(S^3\setminus K(p,q))$ and its parabolic representation as described by Riley in \cite{Riley}. The main trick is to avoid triangulating the complement: instead, we start in Section \ref{Hopf} from the Hopf formula expressing that the ``torus'' made by the meridian and the longitude is a trivial class in $H_2(G,\Z)$. Using Fox calculus, we get in Section \ref{Fox} an explicit 3-chain in the bar complex of $G$ bounding the peripheral torus. Then, it is a standard procedure to get from this either the volume or the Bloch invariant. We explain this in Section \ref{Bloch} and apply it in Section \ref{App}. 
This seemingly new technique is the main interest of the article and can in principle be applied to other families of knots, for instance pretzel knots. 

Let us mention previous work on the same topic: in \cite{Guer}, the authors triangulate the complement in order to prove the existence of a hyperbolic structure but do not provide a ``closed formula'' for their volume: here we take for granted the existence of the hyperbolic structure. In \cite{HLMR}, the authors provide a formula for two-bridge knot orbifolds by deforming the representation of the complement and using a Schläffli formula for the volume. Our techniques should apply to their case but the trace field is more difficult to compute.

{\bf Acknowledgments:} We thank Gregor Masbaum for thorough discussions around this article and Pierre-Vincent Koseleff for his interest, his numerical computations confirming our formulas and his sharing of Figure \ref{conway}.

\section{Commuting meridian and longitude}\label{Hopf}

Let $p,q$ be two coprime odd integers satisfying $0<q<p$. The fundamental group $G$ of the complement of the two-bridge knot $K(p,q)$ admits the following presentation where we have set $\epsilon_i=(-1)^{\lfloor iq/p\rfloor}$ (see \cite{BZH}):
$$G=\langle u,v | wu=vw\rangle, \quad w=u^{\epsilon_1}v^{\epsilon_2}\cdots u^{\epsilon_{p-2}}v^{\epsilon_{p-1}}.$$
Formally, this means that $G=F/R$ where $F$ is the free group generated by $u$ and $v$ and $R$ is the subgroup of $F$ normally generated by $r=wuw^{-1}v^{-1}$. For any word $x$ in the letters $u$ and $v$, we denote by $x^*$ the word obtained by writing the letters in the reverse direction. This operation is an anti-involution in the sense that it satisfies $(xy)^*=y^*x^*$ and $(x^*)^*=x$ for any $x,y\in F$.

Recall that we defined $\ell$ to be the unique integer satisfying $0<k<2p$ and $k=-q^{-1}\mod 2p$.
\begin{lemma}\label{lemmefonda}
The following equality holds in $F$:
$$uw^*v^{-1}(w^*)^{-1}=grg^{-1}, \quad\textrm{ where } g=u^{\epsilon_0}v^{\epsilon_1}\dots u^{\epsilon_{\ell-1}}v^{\epsilon_\ell}.$$
In particular, the anti-involution $x\mapsto x^*$ induces an anti-involution on $G$. 
\end{lemma}
\begin{proof}

Let us first comment how the first point implies the second. 
We compute $r^*=v^{-1}(w^{-1})^*uw^*$: it is conjugated to $uw^*v^{-1}(w^*)^{-1}$, hence to $r$ by the first statement. This shows that $r^*$ belongs to $R$ as we wanted. 
%\end{proof}
%\begin{lemma}
%The elements $u$ and $w^*w$ commute in $G$.
%\end{lemma}
%\begin{proof}
%One computes $w^*wu=w^*vw$ so that the lemma reduces to proving $w^*v=uw^*$. This relation can be deduced from the defining one by applying the $*$-involution. It is then sufficient to show that $r^*$ is in the normal subgroup generated by $r$. We prove now that $r^*$ and $r$ are actually conjugated in the free group generated by $u$ and $v$. 

To prove the first statement, we observe that the relation $r$ is a sequence of alternating $u$ and $v$ with powers given by the following sequence:
$$S:\epsilon_1,\ldots,\epsilon_{p-1},1,-\epsilon_{p-1},\ldots,-\epsilon_1,-1$$
whereas $uw^*v^{-1}(w^*)^{-1}$ corresponds to the sequence:
$$S':1,\epsilon_{p-1},\ldots,\epsilon_1,-1,-\epsilon_1,\ldots,-\epsilon_{p-1}$$
The sequence $\epsilon_i$ has the symmetries: $\epsilon_{i+p}=-\epsilon_{i}$ and $\epsilon_{-i}=(-1)^{\lfloor-iq/p\rfloor}=-\epsilon_i$ if $i\ne 0\mod p$. In particular the sequence $\epsilon_1,\ldots,\epsilon_{p-1}$ is palindromic.
We also observe that the sequence $S'$ is simply the sequence $\epsilon_0,\ldots,\epsilon_{2p-1}$. The lemma is proven if one can  show that the sequence $S$ is obtained from $S'$ by an even shift. 

By definition of $\ell$, there is an integer $s$ such that $\ell q=-1+2ps$. It follows that $\epsilon_{\ell+i}=(-1)^{\lfloor \ell q/p+iq/p\rfloor}=(-1)^{\lfloor iq/p-1/p\rfloor}$. If $i$ is not divisible by $p$ we have $\epsilon_{\ell+i}=\epsilon_i$ otherwise $\epsilon_{\ell+i}=-\epsilon_i$. This shows that starting $S'$ at the index $\ell+1$, one recovers $S$, which proves the lemma.
\end{proof}

It is well-known that a longitude of $K(p,q)$ is given by $l=w^*w$: using Lemma \ref{lemmefonda}, we can prove it from the presentation of $G$. We get precisely the following formula:

\begin{lemma}\label{prodcom}
Set $l=w^*w$. We have in $F$ the identity:
$$ [l,u]=[w^*,r][r,g].$$
\end{lemma}

This formula says more: recall that the Hopf formula identifies $H_2(G,\Z)$ with $[F,F]\cap R/[F,R]$. Here $[l,u]$ represents a class in $H_2(G,\Z)$ which apparently vanishes. 
This fact is topologically obvious because this class corresponds to the boundary torus under the isomorphism $H_2(G,\Z)=H_2(S^3\setminus K(p,q),\Z)$ and this torus bounds the fundamental class of $S^3\setminus K(p,q)$. It is also algebraically clear because the presentation involves only one relation $r$ whose abelianization is non trivial. However, we will need the explicit formula given by the lemma, that we prove now. 

\begin{proof}
We compute:
\begin{eqnarray*} [w^*w,u]&=&w^*wuw^{-1}(w^*)^{-1}u^{-1}=w^*rv(w^*)^{-1}u^{-1}\\
&=&[w^*,r]rw^*v(w^*)^{-1}u^{-1}=[w^*,r][r,g]
\end{eqnarray*}
This element belongs to $[F,R]$, where $F=\langle u,v\rangle$ and $R$ is the subgroup normally generated by $r$, showing the lemma.
\end{proof}

\section{From the Hopf formula to the bar complex}\label{Fox}

We now compare explicitly the Hopf formula for $H_2(G,\Z)$ with the definition of the same group coming from the bar complex. This is a standard exercise, see \cite[Ex 4. p.46]{Brown}: given a formula expressing the vanishing of some class in $H_2(G,\Z)$, it provides an explicit 3-chain that we will use to compute the volume of the knot complement. 

Recall the definition of the bar complex $C_*(G)$. The abelian group $C_k(G)$ is freely generated by $k$-tuples $[g_1|\cdots|g_k]$ with $g_1,\ldots,g_k\in G$. We extend this notation by multilinearity, replacing $g_i$ by elements in $\Z[G]$. We will need the following formulas for the differential where $\epsilon:\Z[G]\to \Z$ is the augmentation map:

$$\partial [g_1]=0,\quad \partial [g_1|g_2]=\epsilon(g_2)[g_1]-[g_1g_2]+\epsilon(g_1)[g_2],$$
$$\partial[g_1|g_2|g_3]=\epsilon(g_1)[g_2|g_3]-[g_1g_2|g_3]+[g_1|g_2g_3]-\epsilon(g_3)[g_1|g_2].$$

Suppose that $G=F/R$ is a presentation where $F$ is freely generated by $x_1,\ldots,x_n$. Following \cite[Proposition 5.4, p.43]{Brown}, there is an exact sequence of $\Z[G]$-modules
$$
\xymatrix{0\ar[r]& R_{\rm ab}\ar[r]^(0.3){\partial}& \bigoplus_{i=1}^n\Z[G] e_i\ar[r]^(0.6){\partial}& \Z[G]\ar[r]^\epsilon & \Z\ar[r]& 0.} 
$$
Here $R_{\rm ab}$ denotes the abelianization of $R$ and the differentials are defined as follows: $\partial (fe_i)=f(x_i-1)$ and for $r\in R$ we have:
$$\partial \overline{r}=\sum_{i=1}^n \frac{\partial r}{\partial x_i} e_i$$
In this formula, for any $f\in F$, $\frac{\partial f}{\partial x_i}\in \Z[G]$ is the Fox derivative: it is defined uniquely by the conditions $\frac{\partial x_j}{\partial x_i}=\delta_{ij}$ and 
$$\frac{\partial (fg)}{\partial x_i}=\frac{\partial f}{\partial x_i}+f\frac{\partial g}{\partial x_i}.$$
This derivative, restricted to $R$ is a group homomorphism, hence factors through $R_{\rm ab}$. 
The fact that $\partial^2=0$ comes from the following fundamental formula, valid for any $f\in F$:
$$f-1=\sum_{i=1}\frac{\partial r}{\partial x_i}(x_i-1).$$

We may compare this to the bar complex as follows:
$$
\xymatrix{&R_{\rm ab}\ar[r]\ar[d]^{\phi_2}& (\Z[G]^n)\ar[r]\ar[d]^{\phi_1} & \Z[G]\ar[d]\ar[d]^{\phi_0} \\
C_3(G)\ar[r]&C_2(G)\ar[r]&C_1(G)\ar[r]&C_0(G)
}
$$

The map $\phi_2$ is given by $\phi_2(f)=\sum_{i=1}^n [\frac{\partial f}{\partial x_i}|x_i]$. The map $\phi_1$ is given by $\phi_1(fe_i)=\epsilon(f)[x_i]$ and $\phi_0=\epsilon$. We leave to the reader checking that $\phi_*$ is a chain map. The maps $\phi_1,\phi_0$ are obviously $G$-invariant, not $\phi_2$. We have indeed $\phi_2(f.r)-\phi_2(r)=\phi_2(frf^{-1})-\phi_2(r)=\phi_2([f,r])$.

One computes 
\begin{equation}\label{phi2}\phi_2([f,r])=\sum_i [\frac{\partial [f,r]}{\partial x_i}|x_i]=\sum_i [(f-1)\frac{\partial r}{\partial x_i}|x_i]=-\sum_i \partial [f|\frac{\partial r}{\partial x_i}|x_i].
\end{equation}

This shows that taking the coinvariants of the first line, one has a morphism of complexes inducing an isomorphism in homology, proving the Hopf formula. 
$$
\xymatrix{R/[F,R]\ar[r]\ar[d]& F_{\rm ab}\ar[r]\ar[d] & \Z\ar[d] \\
C_2(G)/\partial C_3(G)\ar[r]&C_1(G)\ar[r]&C_0(G)
}
$$

This diagram is the key ingredient to prove the following proposition.

\begin{proposition}
Let $G$ be a group with presentation $G=F/R$ where $F$ is the free group generated by $x_1,\ldots,x_n$ and suppose that the following equation holds in $F$ for some $g,h,f_1,\ldots,f_k\in F$ and $r_1,\ldots,r_k\in R$:
\begin{equation}\label{ocl}
\quad\quad\quad\quad\quad\quad\quad\quad [g,h]=\prod_{i=1}^k [f_i,r_i].
\end{equation}

Then we have the equality $[g|h-1]-[h|g-1]=\partial z$ in $C_2(G)$ where 
$$z=\sum_{i=1}^n [g|\frac{\partial h}{\partial x_i}|x_i]-[h|\frac{\partial g}{\partial x_i}|x_i]-\sum_{j=1}^k\sum_{i=1}^n[f_j|\frac{\partial r_j}{\partial x_i}|x_i].$$
\end{proposition}

\begin{proof}
We compute $\phi_2$ on both side of Equation \eqref{ocl}. Applying equation \eqref{phi2} to the right hand side gives directly the double sum of the proposition. 

We can simplify the formula of $\phi_2$ in the case when it is applied to a pair of commuting elements, as in the left hand side of Equation \eqref{ocl}. Suppose $g,h\in F$ satisfy $[g,h]\in R$: then 
$$\phi_2([g,h])=\sum_{i=1}^n[(1-h)\frac{\partial g}{\partial x_i}+(g-1)\frac{\partial h}{\partial x_i}|x_i].$$

As $$\partial \sum_{i=1}^n[g|\frac{\partial h}{\partial x_i}|x_i]=\sum_{i=1}^n[(1-g)\frac{\partial h}{\partial x_i}|x_i]+\sum_{i=1}^n[g|\frac{\partial h}{\partial x_i}(x_i-1)]$$
we get 
$$\phi_2([g,h])+\partial t=\sum_{i=1}^n[g|\frac{\partial h}{\partial x_i}(x_i-1)]-\sum_{i=1}^n[h|\frac{\partial g}{\partial x_i}(x_i-1)]=[g|h-1]-[h|g-1].$$
where $t=\sum_{i=1}^n [g|\frac{\partial h}{\partial x_i}|x_i]-\sum_{i=1}^n [h|\frac{\partial g}{\partial x_i}|x_i]$ and we used the fundamental formula for $g$ and $h$ in the last equality. This finishes the proof of the proposition.
\end{proof}

Let us apply this proposition to the case of $G=\pi_1(S^3\setminus K(p,q))$. Using the formula of Lemma \ref{prodcom}, we have $[l,u]=[w^*,r][g,r]^{-1}$ which gives:
%directly get the following proposition:
%Our interest however is that it produces an explicit 3-chain bounding an element in $H_2(G,\Z)$. Consider for instance the class of $[w^*w,u]\in [F,F]\cap R$. By the equality $[w^*w,u]=[w^*,r][g,r]^{-1}$ one finds 

\begin{eqnarray*}
 \phi_2([w^*w,u])&=&\partial[g-w^*|\frac{\partial r}{\partial u}|u]+\partial[g-w^*|\frac{\partial r}{\partial v}|v]\\
 &=& \partial[g-w^*|(1-v)\frac{\partial w}{\partial u}+w|u]+\partial[g-w^*|(1-v)\frac{\partial w}{\partial v}-1|v].
 \end{eqnarray*}
We sum up the result in the following proposition.
\begin{proposition}\label{calcul}
Let us write $l=w^*w$ we have $[l|u-1]-[u|l-1]=\partial z$ where
$$z=[g-w^*|(1-v)\frac{\partial w}{\partial u}+w|u]+[g-w^*|(1-v)\frac{\partial w}{\partial v}-1|v]+[l|1|u]-[u|\frac{\partial l}{\partial u}|u]-[u|\frac{\partial l}{\partial v}|v].$$
\end{proposition}
\section{From the bar complex to the volume}\label{Bloch}

\subsection{Volume of closed hyperbolic 3-manifolds}
%Let $k$ be the trace field of $K(p,q)$ and $\rho:G\to {\mathrm{SL}}_2(k)$ be a lift of the holonomy representation. It is well-known that one can suppose $$\rho(u)=\begin{pmatrix} 1 & 1 \\ 0 & 1\end{pmatrix},\quad 
%\rho(v)=\begin{pmatrix} 1 & 0 \\ x & 1\end{pmatrix}$$
%for some $x\in k$. 
Let $D$ be the Bloch-Wigner dilogarithm: $D(z)$ is the volume of an ideal tetrahedron in $\H^3$ whose vertices have cross-ratio equal to $z$. 
It is known that given $\infty\in \P^1(\C)$, the $3$-cocycle 
$$c(g_1,g_2,g_3)=D([\infty:g_1\infty:g_1g_2\infty:g_1g_2g_3\infty])$$
represents the volume class in $H^3({\mathrm{PSL}}_2(\C),\R)$. This means that if $M$ is a closed hyperbolic $3$-manifold with holonomy $\rho:\pi_1(M)\to \mathrm{PSL}_2(\C)$, then denoting by $f:M\simeq \mathrm{B}\pi_1(M)\to \mathrm{BPSL}_2(\C)$ a continuous map inducing $\rho$ on fundamental groups, we get:
$$\operatorname{Vol}(M)=\int_M f^* c.$$

In the present article, we perform this computation in $H_3(\pi_1(M),\Z)$. Indeed, knowing a $3$-cycle $z_M$ representing $[M]$ in $C_3(\pi_1(M))$, one can simply compute $\operatorname{Vol}(M)=\langle c, \rho_*z\rangle$. 

The situation here is slightly more involved due to the fact that $M=S^3\setminus K(p,q)$ is not compact: it can be compactified into a $3$-manifold $\overline{M}$ with toric boundary such that the representation $\rho:\pi_1(\overline{M})\to \mathrm{PGL}_2(k)$ takes values in a subfield $k$ of $\C$ and restricted to the boundary, takes its values in $B$, the Borel subgroup of PGL$_2(k)$ stabilizing $\infty$. 

In the next section, we take this difficulty into account. Morever, we replace the volume with a finer invariant, called the Bloch invariant of $M$. For more details, we refer to \cite{NeumannYang}.

\subsection{Bloch group and Bloch invariant}

Let $C_*(\P^1(k))$ be the complex where $C_n(\P^1(k))$ is freely generated by $(n+1)$-tuples of distinct elements in $\P^1(k)$. An element $[z_0,\ldots,z_n]\in C_n(\P^1(k))$ with repeated terms will be considered to be $0$. The differential is defined as usual by $\partial [z_0,\ldots,z_n]=\sum_{i=0}^n (-1)^i [z_0,\ldots,\hat{z_i},\ldots,z_n]$. 

This complex is an acyclic complex of $G$-modules, where $G=\mathrm{PGL}_2(k)$. We denote by $C_*(\P^1(k))_{G}$ the complex of coinvariants: for instance $C_3(\P^1(k))_{G}$ is the free abelian group generated by $[z]$ for $z\in k\setminus\{0,1\}$. The element $[z]$ corresponds to the quadruple $[\infty,0,1,z]$ and the quadruple $[z_0,z_1,z_2,z_3]$ to its cross-ratio

$$[z_0:z_1:z_2:z_3]=\frac{(z_0-z_2)(z_1-z_3)}{(z_0-z_3)(z_1-z_2)}$$

By definition, the pre-Bloch group $\mathcal{P}(k)$ is the cokernel of the map $\partial : C_4(\P^1(k))_G\to C_3(\P^1(k))_G$. Writing $[\infty,0,1,x,y]\in C_4(\P^1(k))$, we find that $\PB(k)$ has generators $[z]$ for $z\in k\setminus\{0,1\}$ and relations 
$$[x]-[y]+[\frac{y}{x}]-[\frac{1-y}{1-x}]+[\frac{1-y^{-1}}{1-x^{-1}}], \quad x,y\in k\setminus\{0,1\}.$$

We define $\delta:\PB(k)\to k^*\tilde{\otimes}k^*$ by $\delta([z])=z\otimes(1-z)$ where $k^*\tilde{\otimes}k^*=k^*\otimes_\Z k^*/(x\otimes y+y\otimes x)$. The Bloch group is defined by $\B(k)=\ker \delta$. It fits in the following exact sequence
$$\xymatrix{
0\ar[r]& \B(k)\ar[r]&\PB(k)\ar[r]^\delta & k^*\tilde\otimes k^*\ar[r]&K_2(k)\ar[r]&0
}$$
We introduce the last term only for completeness: it comes from the presentation of $K_2(k)$ given by Matsumoto's theorem. There is a morphism of complexes 
$$\Phi:
\begin{cases}
C_*(G)\to C_*(\P^1(k))\\
[g_1|\cdots|g_k]\mapsto [\infty,g_1\infty,\ldots,g_1\cdots g_k\infty]
\end{cases}$$
If we denote by $B\subset G$ the Borel subgroup stabilizing $\infty$, it defines a subcomplex $C_*(B)\subset C_*(G)$ on which $\Phi$ vanishes. 

We get in this way an induced map $\Phi:H_3(G,B)\to \PB(k)$ where $H_*(G,B)$ denotes the homology of the quotient complex $C_*(G)/C_*(B)$. By definition, the Bloch invariant of a pair $(\overline{M},\rho)$ where $\overline{M}$ is a 3-manifold with toric boundary and $\rho:\pi_1(\overline{M})\to G$ maps the boundary into $B$ is 
$$\beta(\overline{M},\rho)=\Phi(f_*([\overline{M},\partial\overline{M}]))$$
where $f:(\overline{M},\partial\overline{M})\to (\mathrm{B}G,\mathrm{B}B)$ induces $\rho$. 

We will use the following lemma whose proof can be found in \cite{NeumannYang}. 

\begin{lemma}
Suppose that $\rho:\pi_1(\overline{M})\to \mathrm{PSL}_2(k)$ is such that $\rho(m)=\pm\begin{pmatrix} a & * \\ 0& a^{-1}\end{pmatrix}$ and $\rho(l)=\pm\begin{pmatrix} b & * \\ 0& b^{-1}\end{pmatrix}$. Then $\delta \beta(\overline{M},\rho)=a\tilde\otimes b$ up to $2$-torsion.
\end{lemma}
In particular, if $\rho$ is parabolic on the boundary, then $\beta(\overline{M},\rho)\in \B(k)$ modulo torsion. 
%$$\delta \beta(\overline{M},\rho)=..$$
%\end{lemma}
%\begin{proof}
%To make this computation, we introduce an intermediate complex. Let $V$ be a 2-dimensional $k$-vector space so that $\P^1(k)=\P(V)$ and set $C_*(V)$ to be the free abelian group generated by $(k+1)$-tuples $[v_0,\ldots,v_k]$ of distinct elements of $V\setminus\{0\}$. It is again an acyclic complex endowed with the usual differential. 
%Let $L\subset V$ be line corresponding to $\infty\in \P(V)$. We denote by $C_*(L)$ the subcomplex of $C_*(V)$ consisting of $(k+1)$-tuples belonging to $L$. 

%Set $G'=\mathrm{SL}_2(k)$ and $B'\subset G'$ the stabilizer of $L$. Given $v\in L\setminus \{0\}$, we have a similar morphism of complex $\Phi:C_*(G')\to C_*(V)$ mapping $C_*(B')$ onto $C_*(L)$. We denote by $C_*(V)_{G'}$ the complex of co-invariants and by $C_*(L)_{G'}$ the image of $C_*(L)$ in $C_*(V)_{G'}$. Finally 

%We now look at the diagram:
%$$\xymatrix{
%H_3(G',B')\ar[r]\ar[d]& H_3(V,L)\ar[r]\ar[d]& \PB(k)\ar[d]^\delta\\
%H_2(B')\ar[r]& H_2(L)\ar[r]& k^*\wedge k^*
%}$$

%The map $H_2(B')\to H_2(L)$ sends $[a|b]$ to $[v,av,abv]$, the map $H_2(L)\to k^*\wedge k^*$ maps 
%\end{proof}

%$$V(S^3\setminus K(p,q))=c(g-w^*,(1-v)\frac{\partial w}{\partial v}-1,v).$$

\section{An explicit formula in the Bloch group}\label{App}

We apply the construction of the preceding section to the case of two-bridge knot complements. 
Let us identify the generators $u,v$ of $G=\pi_1(S^3\setminus K(p,q))$ with their image in SL$_2(\C)$ by the holonomy representation. As $u$ and $v$ are meridians, their image are parabolic hence up to conjugation, one can find $x\in \C$ such that 

$$u=\begin{pmatrix} 1 & x \\ 0 & 1\end{pmatrix}\textrm{ and }v=\begin{pmatrix} 1 & 0 \\ x & 1\end{pmatrix}$$

Given $z\in \P^1(\C)$ represented by $(z,1)\in \C^2$, we find that $u(z)=z+x$ and $v(z)=\frac{1}{x+\frac{1}{z}}$. As $wu=vw$, we have $w(\infty)=v(w(\infty))$. As $0$ is the unique fixed point of $v$, we get $w(\infty)=0$ or 
$$\epsilon_1 x+
\cfrac{1}
{\epsilon_2 x+\cfrac{1}{\cdots+\cfrac{1}{\epsilon_{p-1}x}}
}=0.$$

Converesely, if $x$ satisfies the above equation and $k\subset \C$ is the number field generated by $x$, then $u,v\in \mathrm{SL}_2(k)$ satisfy the defining equation $wu=vw$. 

Applying the formula of Proposition \ref{calcul}, we get a formula for the Bloch invariant of $S^3\setminus K(p,q)$ endowed with the above representation. For short, we denote this element by $\beta_{p,q}\in \B(k)$. 

As $[z_0,z_1,z_2,z_3]=0$ if two of the $z_i$'s coincide, there are a lot of simplifications in this formula: each term of the form $[g_1|g_2|g_3]$ is sent to $[\infty\deuxp g_1\infty\deuxp g_1g_2\infty\deuxp g_1g_2g_3\infty]$ so that if any of $g_1,g_2,g_3$ fixes $\infty$, the term vanishes. We finally get only 
$$\beta_{p,q}=\Phi([g-w^*|(1-v)\frac{\partial w}{\partial v}|v]).$$ 

Derivating $\frac{\partial w}{\partial v}$ we get $(p-1)/2$ terms corresponding to the term $v^{\epsilon_{2j}}$ for $j=1,\ldots,(p-1)/2$. If $\epsilon_{2j}=1$, this term is $u^{\epsilon_1}v^{\epsilon_2}\cdots u^{\epsilon_{2j-1}}$ and if $\epsilon_{2j}=-1$, 
it gives $-u^{\epsilon_1}v^{\epsilon_2}\cdots u^{\epsilon_{2j-1}}v^{-1}$. Using the formula $[z_0\deuxp z_1\deuxp z_2 \deuxp z_3]=-[z_0\deuxp z_1\deuxp z_3\deuxp z_2]$ and expanding, we get 

\begin{flalign*}
\beta_{p,q}=&\sum_{j}[\infty, g\infty, gu^{\epsilon_1}\cdots u^{\epsilon_{2j-1}}\infty,gu^{\epsilon_1}\cdots u^{\epsilon_{2j-1}}v^{\epsilon_{2j}}\infty]\\
&-\sum_{j}[\infty, g\infty, gvu^{\epsilon_1}\cdots u^{\epsilon_{2j-1}}\infty,gvu^{\epsilon_1}\cdots u^{\epsilon_{2j-1}}v^{\epsilon_{2j}}\infty]\\
&-\sum_{j}[\infty, w^*\infty, w^*u^{\epsilon_1}\cdots u^{\epsilon_{2j-1}}\infty,w^*u^{\epsilon_1}\cdots u^{\epsilon_{2j-1}}v^{\epsilon_{2j}}\infty]\\
&+\sum_{j}[\infty, w^*\infty, w^*vu^{\epsilon_1}\cdots u^{\epsilon_{2j-1}}\infty,w^*vu^{\epsilon_1}\cdots u^{\epsilon_{2j-1}}v^{\epsilon_{2j}}\infty]
\end{flalign*}

To simplify this formula, we can apply $g^{-1}, (gv)^{-1},(w^*)^{-1},(w^*v)^{-1}$ to all terms in each of the respective lines. We also observe that $l=w^*w$ fixes $\infty$ so that $w^*(w(\infty))=w^*(0)=\infty$ hence $(w^*)^{-1}(\infty)=0$. 

Write $z_0=\infty$ and for $j\ge 1$
$$z_j=\epsilon_1 x+
\cfrac{1}
{\epsilon_2 x+\cfrac{1}{\cdots+\cfrac{1}{\epsilon_{j}x}}
}\textrm { so that }u^{\epsilon_1}v^{\epsilon_2}\cdots v^{\epsilon_{2j}}\infty=z_{2j}.$$

Recall that $g=u^{\epsilon_0}v^{\epsilon_1}\cdots v^{\epsilon_k}$ and $\epsilon_{\ell-i}=\epsilon_{-i}=-\epsilon_i$ provided that $i$ is not divisible by $p$. 
If $\ell<p$, we get $g=u^{-\epsilon_\ell}v^{-\epsilon_{\ell-1}}\cdots v^{-\epsilon_2}u^{-\epsilon_1}v^{-1}$ hence 
$$g^{-1}\infty=vu^{\epsilon_1}v^{\epsilon_2}\cdots u^{\epsilon_\ell}\infty=v(z_{\ell-1}).$$ 

If $\ell>p$, recall from Lemma \ref{lemmefonda} that $u^{\epsilon_0}v^{\epsilon_1}\cdots v^{\epsilon_{2p-1}}=1$ %so that $g u^{\epsilon_{k+1}}\cdots v^{\epsilon_{2p-1}}=1$ 
which implies $g^{-1}=u^{\epsilon_1}v^{\epsilon_2}\cdots v^{\epsilon_{2p-1-\ell}}$. This finally gives $g^{-1}(\infty)=z_{2p-1-\ell}$.

A nice way to synchronize these two formulas is to notice that we also have $u^{\epsilon_1}\cdots v^{\epsilon_{2p}}=1$. Hence for even $n<p$: $z_{2p-n}=u^{\epsilon_1}\cdots v^{\epsilon_{2p-n}}\infty=(u^{\epsilon_{2p-n+1}}\cdots v^{\epsilon_{2p}})^{-1}\infty=v^{-1}u^{\epsilon_1}\cdots u^{\epsilon_{n-1}}\infty=v^{-1}(z_{n-2})$. 
This shows that the formula for $k>p$ is the same as for $k<p$ and moreover, we get $z_{2p-2}=v^{-1}(\infty)$. Hence we can write in any case:
\begin{equation*}
    \begin{split}
\beta_{p,q}=\sum_{j}\Big([v(z_{\ell-1})\deuxp \infty\deuxp z_{2j-2}\deuxp z_{2j}]-[z_{\ell-1}\deuxp z_{2p-2}\deuxp z_{2j-2}\deuxp z_{2j}]-[0\deuxp \infty\deuxp z_{2j-2}\deuxp z_{2j}]\\
+[0\deuxp z_{2p-2}\deuxp z_{2j-2}\deuxp z_{2j}]\Big)
\end{split}
\end{equation*}
which can be reordered as follows:

\begin{equation*}
\begin{split}
\beta_{p,q}=\sum_{j}\Big([v(z_{\ell-1})\deuxp \infty\deuxp z_{2j-2}\deuxp z_{2j}]+[\infty\deuxp 0\deuxp z_{2j-2}\deuxp z_{2j}]
+[0\deuxp z_{2p-2}\deuxp z_{2j-2}\deuxp z_{2j}]\\
+[z_{2p-2}\deuxp z_{\ell-1}\deuxp z_{2j-2}\deuxp z_{2j}]\Big)
\end{split}
\end{equation*}
Computing explicitly the cross-ratios, we get the formula of Theorem \ref{principal}.

\begin{remark}
Consider the cases $p=2n+1,q=1$ so that $\ell=2p-1$ and $z_{\ell-1}=z_{2p-2}$. We observe that $\epsilon_1=\cdots=\epsilon_{p-1}=1$ so that the family of polynomials $P_j$ is a variant of the Chebyshev polynomial. In particular, its roots are $2i\cos(\pi j/p)$ for $j=1,\ldots,p-1$. In the last formula for $\beta_{p,q}$, all elements belong to the circle of imaginary numbers, so that their cross ratio is real. This shows that the volume of these knots vanishes, as expected. One can observe that their Bloch invariant also vanishes, as they belong to a totally real number field whose Bloch group vanishes thanks to Borel's theorem, see \cite{NeumannYang}.  
\end{remark}

%\subsubsection{Twist knots}
%Here we set $p=2n+1$ and $q=2n-1$ so that $\kappa=n+1$ if $n$ is even and $\kappa=n$ if $n$ is odd. In this case, the polynomial $P_{p-1}$ is known to be irreducible so that $\Q[x]/(P_{p-1})$ is the trace field of $P$. The root corresponding to the hyperbolic representation is the one with highest real part. As it is well-known, their volume converges to the volume of the Whitehead link which is approximately $3.66$.

\section{A formula for $\zeta_k(2)$?}

This paragraph is of more speculative nature and asks whether one can compute explicitely $\zeta_k(2)$ where $k$ is the trace field of $K(p,q)$ and $\zeta_k$ is the Dedekind zeta function. For arithmetic varieties, the volume is known to be proportional to $\zeta_k(2)$ where the proportionality factor depends on arithmetic invariants, see \cite[Chapter 11]{MR}.
Among two-bridge knots, only the figure-eight knot $K(5,3)$ is arithmetic, yielding the well-known formula 
$$ V(5,3)=6\Lambda(\pi/3)=3D(j)=\frac{9\sqrt{3}}{\pi^2}\zeta_{\Q(j)}(2)$$
where $\Lambda$ is the Lobatchevsky function and $j=\exp(2i\pi/3)$. 

In this case, one has $P_{p-1}=x^4-x^2+1$, $z_0=\infty,z_2=x-1/x=i,z_4=0$, $\ell=3$. All terms in the formula vanish except two terms: $[v(z_2):\infty:z_2:z_4]=[(x-i)^{-1}:\infty:i:0]=[1-x^2]=[-j^2]$ and 
$[0:z_8:z_0:z_2]=[0:-1/x:\infty:i]=[x^{-2}]=[-j^2]$. Applying $D$ and summing the terms yield the formula for $V(5,3)$.

A more subtle example is given by the knot $K(7,3)$ (the knot $5.2$ in Rolfsen's table) whose volume is approximately $2,82812208833078$ and three times the volume of the Weeks manifold, hence directly related to $\zeta_k(2)$ where $k=\Q[t]/(t^3-t-1)$. Precisely we have for $t$ the complex root of $t^3-t-1$ with positive imaginary part

$$V(7,3)=6D(t)=\frac{9(23)^{3/2}}{4\pi^4}\zeta_k(2).$$

We cannot expect that such formulas still hold for bigger values of $p$, however it is known that $\zeta_k(2)$ can be expressed with products of values of $D$ on algebraic numbers, see \cite{Zagier}.

Numerical experiments show that the trace fields of $K(p,q)$ tend to have class number equal to one and a lot of complex embeddings (in the case $q>1$). For instance in \cite{moi}, the author proved that the number of real embeddings of $k$ is more than $\frac{1}{2}|\sum_{j=1}^{p-1}\epsilon_j|$. Let $\mathcal{O}$ be the ring of integers of $k$: the number $\zeta_k(2)$ is related to the volume of the quotient $(\H^2)^{r_1}\times (\H^3)^{r_2}/\SL_2(\mathcal{O})$ where $d=r_1+2r_2$ is the dimension of $k$ over $\Q$. The cusp corresponding to $\infty$ has stabilizer equal to the Borel subgroup $B\subset \SL_2(\mathcal{O})$, itself commensurable with the semi-direct product $\mathcal{O} \rtimes \mathcal{O}^\times$. This means that, provided that the class number is equal to 1, the corresponding fundamental class of the cusp in  $H_{d}(B,\Z)$ vanishes in $H_d(\SL_2(\mathcal{O}),\Z)$. Finding an explicit $(d+1)$-chain bounding it would yield a formula for $\zeta_k(2)$ in the spirit of the present article. Of course, our techniques should be adapted as they are available only in low dimensions.


\begin{thebibliography}{10}
\bibitem[B]{Brown}
K. S. Brown,
\newblock Cohomology of groups
\newblock Graduate Texts in Mathematics, Springer.

\bibitem[BZH]{BZH}
G.Burde, H. Zieschang, and M. Heusener
\newblock Knots. 3rd edition.
\newblock De Gruyter Studies in Mathematics 5. Berlin (2014).

\bibitem[G]{Guer}
F. Guéritaud,
\newblock On canonical triangulations of once-punctured torus bundles and two-bridge link complements. With an appendix by David Futer.
\newblock \emph{Geom. Topol.} 10, 1239-1284 (2006).

\bibitem[HLMR]{HLMR}
J-Y. Ham, J. Lee, A. Mednykh and A. Rasskazov
\newblock On the volume and Chern-Simons invariant for 2-bridge knot orbifolds.
\newblock \emph{J. Knot Theory Ramifications} 26, No. 12, (2017).

\bibitem[L]{Lackenby}
M. Lackenby,
\newblock The volume of hyperbolic alternating link complements, with an appendix by Ian Agol and Dylan Thurston. 
\newblock \emph{Proc. Lond. Math. Soc.}, III. Ser. 88, No. 1, 204-224 (2004).

\bibitem[MR]{MR}
C. Maclachlan and A. W. Reid,
\newblock The arithmetic of hyperbolic 3-manifolds.
\newblock \emph{Graduate Texts in Mathematics.} 219. New York, (2003).

\bibitem[M]{moi}
J. Marché,
\newblock Signatures of TQFTs and trace fields of two-bridge knots
\newblock arXiv:2309.03656

\bibitem[NY]{NeumannYang}
W. D. Neumann and J. Yang,
\newblock Bloch invariants of hyperbolic 3-manifolds.
\newblock \emph{Duke Math. J.} 96, No. 1, 29-59 (1999).

\bibitem[R]{Riley}
R. Riley,
\newblock Parabolic representations of knot groups. I.
\newblock \emph{Proc. Lond. Math. Soc.}, III. Ser. 24, 217-242 (1972).

\bibitem[Z86]{Zagier}
D. Zagier,
\newblock Hyperbolic manifolds and special values of Dedekind zeta-functions.
\newblock \emph{Invent. Math.},  83, 285-301 (1986).

\bibitem[Z07]{dilog}
D. Zagier,
\newblock The dilogarithm function. In 
\newblock \emph{Frontiers in number theory, physics, and geometry}. II, pages
3–65. Springer, Berlin, 2007.
\end{thebibliography}
\end{document}